\documentclass[a4paper,reqno]{amsart}
\usepackage[english]{babel}
\usepackage{amsmath,amsthm,amssymb,amsfonts}
\usepackage{mathrsfs}

\usepackage{enumitem}
\setlist[enumerate]{leftmargin=*}

\usepackage{hyperref}

\allowdisplaybreaks

\newtheorem{thm}{Theorem}[section]
\newtheorem{prp}[thm]{Proposition}
\newtheorem{cor}[thm]{Corollary}
\newtheorem{lem}[thm]{Lemma}
\theoremstyle{definition}
\newtheorem{dfn}[thm]{Definition}
\newtheorem{rem}[thm]{Remark}

\numberwithin{equation}{section}

\newcommand{\tc}{\,:\,}
\newcommand{\defeq}{\mathrel{:=}}

\newcommand{\HLM}{\mathcal{M}}

\newcommand{\RR}{\mathbb{R}}
\newcommand{\CC}{\mathbb{C}}
\newcommand{\NN}{\mathbb{N}}
\newcommand{\ZZ}{\mathbb{Z}}

\newcommand{\Npos}{\mathbb{N}_+}
\newcommand{\Rpos}{\mathbb{R}^+}
\newcommand{\Rnon}{\mathbb{R}^+_0}
\newcommand{\schur}{\odot}

\newcommand{\pot}{\mathcal{P}}
\newcommand{\halfpot}{\mathcal{HP}}

\newcommand{\mm}{\mathrm{m}}

\newcommand{\uc}{\mathrm{uc}}
\newcommand{\cv}{\mathrm{cv}}

\newcommand{\matP}{\mathbf{P}}
\newcommand{\matA}{\mathbf{A}}

\newcommand{\matM}{\mathbf{M}}
\newcommand{\matN}{\mathbf{N}}
\newcommand{\matF}{\mathbf{F}}

\newcommand{\chr}{\mathbf{1}}

\DeclareMathOperator{\diag}{diag}
\DeclareMathOperator{\inc}{inc}

\DeclareMathOperator{\Card}{\#}
\DeclareMathOperator{\supp}{supp}
\DeclareMathOperator*{\esssup}{ess\,sup}

\newcommand{\sobolev}[2]{L^{#2}_{#1}}

\newcommand{\opL}{\mathcal{L}}
\newcommand{\opH}{\mathcal{H}}

\newcommand{\Kern}{\mathcal{K}}

\newcommand{\inv}{\leftarrow}

\DeclareFontFamily{U}{matha}{\hyphenchar\font45}
\DeclareFontShape{U}{matha}{m}{n}{
      <5> <6> <7> <8> <9> <10> gen * matha
      <10.95> matha10 <12> <14.4> <17.28> <20.74> <24.88> matha12
      }{}
\DeclareSymbolFont{matha}{U}{matha}{m}{n}
\DeclareFontSubstitution{U}{matha}{m}{n}

\DeclareFontFamily{U}{mathx}{\hyphenchar\font45}
\DeclareFontShape{U}{mathx}{m}{n}{
      <5> <6> <7> <8> <9> <10>
      <10.95> <12> <14.4> <17.28> <20.74> <24.88>
      mathx10
      }{}
\DeclareSymbolFont{mathx}{U}{mathx}{m}{n}
\DeclareFontSubstitution{U}{mathx}{m}{n}

\DeclareMathDelimiter{\vvvert}{0}{matha}{"7E}{mathx}{"17}

\renewcommand{\vartheta}{\gamma}

\begin{document}

\title[Grushin operators in the plane]{An optimal multiplier theorem for Grushin operators in the plane, II}

\author{Gian Maria Dall'Ara}
\address[G.\ M.\ Dall'Ara]{
Istituto Nazionale di Alta Matematica ``Francesco Severi'' \\ 
Research Unit Scuola Normale Superiore \\ Piazza dei Cavalieri 7 \\ 56126 Pisa \\ Italy}
\email{dallara@altamatematica.it}
\author{Alessio Martini}
\address[A.\ Martini]{School of Mathematics \\ University of Birmingham \\ Edgbaston \\ Birmingham \\ B15 2TT \\ United Kingdom}
\email{a.martini@bham.ac.uk}

\thanks{The authors are members of the Gruppo Nazionale per l'Analisi Matematica, la Probabilit\`a e le loro Applicazioni (GNAMPA) of the Istituto Nazionale di Alta Matematica (INdAM)}


\begin{abstract}
In a previous work we proved a spectral multiplier theorem of Mihlin--H\"ormander type for two-dimensional Grushin operators $-\partial_x^2 - V(x) \partial_y^2$, where $V$ is a doubling single-well potential, yielding the surprising result that the optimal smoothness requirement on the multiplier is independent of $V$.
Here we refine this result, by replacing the $L^\infty$ Sobolev condition on the multiplier with a sharper $L^2$ condition. As a consequence, we obtain the sharp range of $L^1$ boundedness for the associated Bochner--Riesz means. The key new ingredient of the proof is a precise pointwise estimate in the transition region for eigenfunctions of one-dimensional Schr\"odinger operators with doubling single-well potentials.
\end{abstract}

\maketitle

\section{Introduction}

\subsection{Statement of the results}
In this paper we continue the analysis, begun in \cite{DM2}, of two-dimensional Grushin operators
\begin{equation}\label{eq:grushin}
\opL = -\partial_x^2 - V(x) \partial_y^2,
\end{equation}
where $V : \RR \to [0,\infty)$ is a ``single-well potential'' satisfying a scale-invariant regularity condition of order $1+\theta$. More precisely, we assume that $V$ is continuous, not identically zero, $C^1$ off the origin, and that, for some $\theta \in (0,1)$, the estimates
\begin{subequations}\label{eq:2dassumptions}
\begin{equation}\label{eq:2dassumptions_doubling}
V(-x) \simeq V(x) \simeq x V'(x),  
\end{equation}
\begin{equation}\label{eq:2dassumptions_holder}
|V'(xe^h)-V'(x)| \lesssim |V'(x)| \, |h|^\theta
\end{equation}
\end{subequations}
hold for all $x \in \RR \setminus\{0\}$ and $h \in [-1,1]$. 
Here we use the standard notation $A \lesssim B$ to denote the estimate $A \leq C B$ for some positive constant $C$, and $A \simeq B$ to denote the conjunction of $A \lesssim B$ and $B \lesssim A$; below we will also write $A \lesssim_s B$ or $A \simeq_s B$ to indicate that the implicit constants may depend on a parameter $s$.
We refer to the introduction of \cite{DM2} for a discussion of the scope of the assumptions \eqref{eq:2dassumptions}; here we limit ourselves to pointing out that they are satisfied by power laws $V(x)=|x|^d$ of any degree $d>0$ and appropriate perturbations thereof.

In \cite{DM2} we proved a spectral multiplier theorem of Mihlin--H\"ormander type for $\opL$, whose smoothness requirement is independent of $V$ and formulated in terms of an $L^\infty$ Sobolev norm of order $s>2/2$, that is, half the topological dimension of the underlying manifold.
The independence from $V$ of the smoothness requirement is particularly striking when compared, e.g., to the classical results based on heat kernel bounds \cite{hebisch_functional_1995,duong_plancherel-type_2002,robinson_analysis_2008}, which would give instead a condition of order $s > (2+d/2)/2$ in the case $V(x) \simeq |x|^d$.
We refer to the introduction of \cite{DM2} for an extensive discussion of the relevance of such result, in the context of a programme  (see also \cite{martini_mueller_golo}) aimed at understanding the optimal smoothness requirement in multiplier theorems for sub-elliptic operators.

While the result of \cite{DM2} is optimal, in the sense that the smoothness threshold $2/2$ cannot be lowered, it is still possible to refine it, by replacing the $L^\infty$ Sobolev norm with an $L^2$ Sobolev norm. This is the main result of the present paper. We write $L^q_s(\RR)$ to denote the $L^q$ Sobolev space of (fractional) order $s$ on $\RR$.

\begin{thm}\label{thm:main}
Let $\opL$ be the Grushin operator \eqref{eq:grushin} on $\RR^2$, where the coefficient $V$ satisfies the estimates \eqref{eq:2dassumptions}. Let $s>2/2$.
\begin{enumerate}[label=(\roman*)]
\item\label{en:main_l1} For all $\mm : \RR \to \CC$ such that $\supp \mm \subseteq [-1,1]$,
\[
\sup_{t > 0} \| \mm(t\opL) \|_{L^1 \to L^1} \lesssim_s \| \mm \|_{\sobolev{s}{2}}.
\]
\item\label{en:main_mh} Let $\eta \in C^\infty_c((0,\infty))$ be nonzero. For all $\mm : \RR \to \CC$ and $p \in (1,\infty)$,
\[
\| \mm(\opL) \|_{L^1 \to L^{1,\infty}} \lesssim_s \sup_{t>0} \| \mm(t \cdot) \eta \|_{\sobolev{s}{2}}, \qquad \| \mm(\opL) \|_{L^p \to L^p} \lesssim_{s,p} \sup_{t>0}\| \mm(t\cdot) \eta \|_{\sobolev{s}{2}}.
\]
\end{enumerate}
\end{thm}

To appreciate the nature of the improvement, one may notice that Theorem \ref{thm:main} gives the sharp $L^1$ boundedness range for Bochner--Riesz means associated with the Grushin operator $\opL$, a result that cannot be deduced from the multiplier theorem of \cite{DM2}.

\begin{cor}\label{cor:bochnerriesz}
Under the same assumptions as in Theorem \ref{thm:main},
the Bochner--Riesz means $(1-r \opL)_+^\lambda$ of order $\lambda$ associated with $\opL$ are bounded on $L^1(\RR^2)$ uniformly in $r \geq 0$ whenever $\lambda > 1/2$.
\end{cor}

The sharpness of Theorem \ref{thm:main} and Corollary \ref{cor:bochnerriesz} follows by a standard ``transplantation'' technique (cf.\ \cite{mitjagin_divergenz_1974,kenig_divergence_1982}; see also \cite[Theorem 5.2]{martini_crsphere}).
Indeed $\opL$ is elliptic (its principal symbol is a positive definite quadratic form) where $x \neq 0$, and therefore the ranges of indices $s$ and $\lambda$ for which the boundedness results in Theorem \ref{thm:main} and Corollary \ref{cor:bochnerriesz} hold cannot be larger than the analogous ranges when $\opL$ is replaced by the Euclidean Laplacian $-\partial_x^2 - \partial_y^2$ on $\RR^2$.

Theorem \ref{thm:main} is already known under more restrictive assumptions on $V$. Namely, the case $V(x) = x^2$ is in \cite{martini_grushin_2012,martini_sharp_2014} and the case $V(x) = |x|$ is in \cite{chen_sharp_2013}. Moreover, in a previous joint paper \cite{dallara_martini}, we established the same result when $V$ is convex, $C^3$ off the origin, and, for some $d\in (1,2]$, the estimates
\[
|x^2 V''(x)| + |x^3 V'''(x)|\lesssim x V'(x) \simeq V(x)= V(-x) \simeq |x|^d 
\]
hold for all $x\in \RR\setminus\{0\}$. This appears to have been the first optimal multiplier theorem for a nonelliptic (sub-elliptic) operator enjoying some form of stability under perturbations of the coefficients of the operator. However, the restriction on the power $d$ cannot be removed using the methods of \cite{dallara_martini}, and the desire to overcome this limitation has been the main motivation for the development of a new proof strategy in \cite{DM2} and in the present paper. Notice that the aforementioned works \cite{robinson_analysis_2008,martini_grushin_2012,martini_sharp_2014,chen_sharp_2013,dallara_martini} treat also higher-dimensional cases, and, as a matter of fact, some higher-dimensional cases could be treated by adapting the methods used here too. However, in the same spirit as in \cite{DM2}, here we consider only two-dimensional Grushin operators.

\subsection{Strategy of the proof}

In order to present the main ideas of the paper, it is convenient to recall the notation for the classes of single-well potentials defined in \cite[Definitions 7.5 and 8.3]{DM2}, which express the assumptions \eqref{eq:2dassumptions} in a quantitative form.

\begin{dfn}
Let $\kappa\geq 1$ and $\theta \in (0,1)$. We denote by $\pot_1(\kappa)$ the class of non-identically zero continuous functions $V : \RR \to [0,\infty)$ which are $C^1$ off the origin and such that
\[
\kappa^{-1} V(x) \leq x V'(x) \leq \kappa V(x), \qquad V(-x) \leq \kappa V(x)
\]
for all $x \neq 0$. We denote by $\pot_{1+\theta}(\kappa)$ the class of the $V \in \pot_1(\kappa)$ that satisfy the additional inequality
\[
|V'(e^h x) - V'(x)| \leq \kappa |h|^\theta
\]
for all $x \neq 0$ and $h \in [-1,1]$.
\end{dfn}

As in other works on the subject, Theorem \ref{thm:main} will be deduced from an appropriate ``weighted Plancherel estimate''. In the present case, in light of \cite[Theorem 4.1]{DM2}, it will be enough to prove that for all $V \in \pot_{1+\theta}(\kappa)$, $\vartheta \in [0,1/2)$, $r>0$, and all continuous functions $\mm : \RR \to \CC$ with $\supp \mm \subseteq [1/4,1]$,
\begin{multline}\label{eq:weighted_plancherel}
\esssup_{z' \in \RR^2} \, r^{2-2\vartheta} \max\{V(r),V(x')\}^{1/2-\vartheta} \int_{\RR^2} |y-y'|^{2\vartheta} \left|\Kern_{\mm(r^2 \opL)}(z,z') \right|^2 \,dz \\
\lesssim_{\theta,\kappa,\vartheta} \|\mm\|_{\sobolev{\vartheta}{2}}^2. 
\end{multline}
Here $z=(x,y)$ and $z'=(x',y')$, while $\Kern_{\mm(r^2 \opL)}$ denotes the integral kernel of the operator $\mm(r^2 \opL)$. Indeed, the estimate \eqref{eq:weighted_plancherel} proves assumption (A) of \cite[Theorem 4.1]{DM2} for $q=2$, while assumption (B) is already proved in \cite[Theorem 9.1]{DM2}. We point out that, in the special case $V(x)=x^2$, the above estimate is proved in \cite{martini_sharp_2014}, while the techniques of \cite{martini_grushin_2012, chen_sharp_2013, dallara_martini} lead to a different Plancherel estimate, with a weight depending only on $x,x'$ in place of $|y-y'|^{2\vartheta}$ and $L^2$ in place of $\sobolev{\vartheta}{2}$ in the right-hand side.

Our proof of the weighted Plancherel estimate \eqref{eq:weighted_plancherel} largely follows the lines of the analogous estimate proved in \cite[Theorem 9.1]{DM2}, with the addition of a key new ingredient: universal
 pointwise estimates for eigenfunctions of one-dimensional Schr\"odinger operators with potentials in the class $\pot_{1+\theta}(\kappa)$.
As in \cite[Section 7]{DM2}, we consider the Schr\"odinger operator $\opH[V] \defeq -\partial_x^2 + V$ on $\RR$ with potential $V\in\pot_{1+\theta}(\kappa)$, and we denote by $E_n(V)$ and $\psi_n(\cdot;V)$ ($n\geq 1$) the corresponding eigenvalues and normalised eigenfunctions. The eigenfunction estimates that we need here have the form
\begin{equation}\label{universal_eigenfunction}
|\psi_n(x;V)| \lesssim |\{ V \leq E_n(V) \}|^{-1/2}  \min\{ n^{\delta/2} , E_n(V)^{\beta/2} |V(x)-E_n(V)|^{-\beta/2} \} 
\end{equation}
for some $\delta,\beta \in (0,1)$, and they have the crucial feature that the implicit constant depends only on $\kappa$ and $\theta$ and not on the specific potential $V$.
The ``universality'' of an estimate such as \eqref{universal_eigenfunction} lies in the fact that the right-hand side is simply expressed in terms of natural quantities such as $V,E_n(V),n$ and universal exponents $\delta,\beta$, and does not depend, e.g., on the degree of polynomial growth of $V$.

In the regions where $V \ll E_n(V)$ and $V \gg E_n(V)$, the estimate \eqref{universal_eigenfunction} is already contained in estimates proved in \cite{DM2}, which actually hold for all $V \in \pot_1(\kappa)$. What is crucial for our present purposes is that \eqref{universal_eigenfunction} also covers the ``transition region'' $\{ V \simeq E_n(V) \}$, where the eigenfunction $\psi_n(\cdot;V)$ exhibits a change in behaviour from oscillatory to decaying.
Various techniques are available to deal with the more general problem of approximating eigenfunctions in the transition region (e.g., Olver's method \cite[Chapter 11]{olverbook} and the WKB method, both yielding approximations in terms of the Airy function), but it does not seem possible to use any of them as a black box to prove \eqref{universal_eigenfunction} in the required generality. The method used here is in fact substantially different and of a more direct nature, establishing the upper bound \eqref{universal_eigenfunction} via a monotonicity argument inspired by what is dubbed the ``Sonin's function'' method in \cite{krasikov}, which in turn refers it back to the work of Szeg\H{o} on orthogonal polynomials \cite[\S 7.31 and \S 7.6]{szego}.

The importance of the pointwise estimate \eqref{universal_eigenfunction} is that from it one can deduce a variant of the ``spectral projector bound'' proved in \cite[Theorem 8.5]{DM2}, which plays a fundamental role in the proof of the weighted Plancherel estimate \eqref{eq:weighted_plancherel}.
Specifically, the desired spectral projector bound (Theorem \ref{thm:new_estimate} below) is obtained by summing instances of the eigenfunction estimate \eqref{universal_eigenfunction} corresponding to different values of $n$ and suitably scaled versions $\tau V$ of the potential $V$, where the scaling parameter $\tau$ depends on $n$. In order to bound the resulting sum, another important ingredient is an approximated Bohr--Sommerfeld identity with logarithmic error term (Proposition \ref{prp:BS_log} below) valid for Schrödinger operators with potentials in the class $\pot_1(\kappa)$, which provides precise information on the ``gaps'' between the quantities $E_n(\tau V)$ involved in the estimate.

\subsection{Structure of the paper}
%
In Section \ref{s:projector_bound} we prove the spectral projector bound in a conditional form, namely, by assuming that suitable pointwise eigenfunction estimates of the form \eqref{universal_eigenfunction} hold.

Section \ref{s:pointwiseeigenfunction} is devoted to the proof of the required pointwise eigenfunction estimates. As discussed in that section, suitable pointwise estimates can be proved for a larger class than $\pot_{1+\theta}(\kappa)$. Indeed, several variants of the above eigenfunction estimates \eqref{universal_eigenfunction} are discussed, which may be of independent interest, with different values of $\delta$ and $\beta$ corresponding to different assumptions on the potential $V$.

Finally, in Section \ref{s:weightedplancherel}, we prove the weighted Plancherel estimate \eqref{eq:weighted_plancherel} with $L^2$ Sobolev norm, which, in light of \cite[Theorem 4.1]{DM2}, implies our main result.

\subsection{Notation}
$\chr_A$ denotes the characteristic function of the set $A$. We set $\Rpos = (0,\infty)$ and $\Rnon = [0,\infty)$. $\NN$ denotes the set of natural numbers (including zero), while $\Npos = \NN \setminus \{0\}$ is the set of the positive integers.
For an invertible function $V$, we write $V^\inv$ to denote its compositional inverse.
$\Card I$ denotes the number of elements of a finite set $I$.
For a measurable subset $A \subseteq \RR$ we denote by $|A|$ its Lebesgue measure.
We write $\Kern_T$ to denote the integral kernel of the operator $T$.

\section{A variant of the spectral projector bound}\label{s:projector_bound}

\subsection{Summary of the results}

As before, let $E_n(V)$ and $\psi_n(\cdot;V)$ ($n\geq 1$) be the eigenvalues and normalised eigenfunctions of the Schr\"odinger operator $\opH[V] = -\partial_x^2 + V$ on $\RR$. 
We begin by recording an immediate consequence of the ``virial theorem'' in \cite[Theorem 7.3]{DM2}. Under more restrictive assumptions on $V$, analogous estimates can be found in \cite[eq.\ (5.5)]{dallara_martini}.

\begin{prp}\label{prp:inv_eigenvalue}
Let $V \in \pot_1(\kappa)$ and $n \in \Npos$. Then the function
\[
\Rpos \ni \tau \mapsto E_n(\tau V) \in \Rpos
\]
is a strictly increasing, real analytic bijection, and
\[
E_n(\tau V) \lesssim_\kappa \tau \partial_\tau E_n(\tau V) \leq E_n(\tau V)
\]
for all $\tau \in \Rpos$. Moreover, if $\Xi_n(\cdot;V) : \Rpos \to \Rpos$ denotes its inverse, then
\[
\Xi_n(\lambda;V) \simeq_{\kappa} \lambda \partial_\lambda \Xi_n(\lambda;V).
\]
for all $\lambda \in \Rpos$.
\end{prp}

The aim of this section is the proof of the following bound, which should be compared to the ``spectral projector bound'' of \cite[Theorem 8.5]{DM2}. 

\begin{thm}\label{thm:new_estimate}
Let $\kappa,a > 1$ and $\theta,\delta \in (0,1)$. Let $\tilde\pot$ be a subcone of $\pot_1(\kappa)$ such that the eigenfunction estimate
\begin{equation*}
|\psi_n(x;V)| \leq a \, |\{ V \leq E_n(V) \}|^{-1/2}  \min\{ n^{\delta/2} , E_n(V)^{\theta/2} |V(x)-E_n(V)|^{-\theta/2} \} 
\end{equation*}
holds for all $V \in \tilde\pot$, $n \in \Npos$, and $x \in \RR$. 
Then, for all $V \in \tilde\pot$ and $\lambda,A \in \Rpos$,
\[
\sum_{\substack{n \in \Npos \\ \lambda/\Xi_n(\lambda;V) \in [A,2A] }} \psi_n(x; \Xi_n(\lambda;V) V)^2 \lesssim_{\kappa,a,\theta,\delta} \lambda^{1/2} (\chr_{V \leq 8 A} + e^{-c \lambda^{1/2} |x|} \chr_{V > 8A}),
\]
where $c=c(\kappa)$.
\end{thm}

The main difference between the previous result and \cite[Theorem 8.5]{DM2} is that the above sum involves eigenfunctions corresponding to different potentials (that is, potentials $\tau V$ where $\tau$ depends on the summation index $n$), so cannot be immediately related to properties of the spectral decomposition of a single Schr\"odinger operator. A similar bound can be found in \cite[Proposition 5.8]{dallara_martini}, under more restrictive assumptions on $V$.

The rest of the section is devoted to the proof of Theorem \ref{thm:new_estimate}.

\subsection{A summation lemma}
The following elementary summation lemma will be a key tool in the proof of the spectral projector bound.

\begin{lem}\label{lem:sum_int}
Let $c \in \Rpos$, $\kappa \in [1,\infty)$, $\theta,\beta \in [0,1)$. Let $I \subseteq \Npos$ and, for all $n \in I$, let $t_n \in [\kappa^{-1},\infty)$ be such that
\begin{equation}\label{eq:almost_gaps}
|t_n-cn| \leq \kappa n^{\beta}.
\end{equation}
Then
\begin{equation}\label{eq:sum_est}
\sup_{\substack{a > 0 \\ 0 < b \leq \kappa a}} \sum_{\substack{n \in I \\  t_n \leq \kappa a}} \min\{ a^{\theta-1} | t_n - b |^{-\theta}, a^{-\beta} \} \lesssim_{\kappa,c,\theta,\beta} 1.
\end{equation}
\end{lem}

The proof of Lemma \ref{lem:sum_int} should be compared to that of \cite[Lemma 10]{martini_sharp_2014}. In the case $\beta=0$, the condition \eqref{eq:almost_gaps} implies that the $t_n$ are essentially equispaced, and the estimate \eqref{eq:sum_est} could be obtained, e.g., by using \cite[Lemma 5.7]{dallara_martini} to estimate a sum with the corresponding integral. The point of this lemma is to show that the a similar estimate can be obtained even when $\beta>0$, that is, under a weaker assumption on the gaps between the $t_n$, by taking advantage of the stronger uniform bound $a^{-\beta}$ in the left-hand side of \eqref{eq:sum_est}.

\begin{proof}[Proof of Lemma \ref{lem:sum_int}]
Note that $t_n,n \gtrsim_\kappa 1$ for all $n \in I$. Hence, from the assumption \eqref{eq:almost_gaps} and the fact that $\beta<1$, we deduce that $t_n \simeq_{\kappa,c,\beta} n$ for all $n \in I$.

For a given $a > 0$, from the condition $t_n \leq \kappa a$ and $t_n \simeq_{\kappa,c,\beta} n$, we deduce that $n \lesssim_{\kappa,c,\beta} a$ as well. Therefore, if $I_a = \{ n \in I \tc t_n \leq \kappa a\}$, then \eqref{eq:almost_gaps} implies that
\[
|t_n-cn| \leq E
\]
for all $n \in I_a$, where $E = E(\kappa,c,\beta,a) \lesssim_{\kappa,c,\beta} a^\beta$. 

We now split $I_a$ into the three subsets
\begin{gather*}
I_- = \{ n \in I_a \tc cn+E < b-c \}, \\
 I_+ = \{ n \in I_a \tc cn-E > b+c \}, \\
I_0 = I_a \setminus (I_+ \cup I_-).
\end{gather*}
Then, for all $n \in I_-$,
\[
t_n \leq cn + E < b-c,
\]
and therefore
\[
| t_n - b |^{-\theta} \leq \inf_{t \in [cn+E,cn+E+c] } |t-b|^{-\theta} \leq \frac{1}{c} \int_{cn+E}^{cn+E+c} |t-b|^{-\theta} \,dt
\]
(here we use that $t \mapsto |t-b|^{-\theta}$ is increasing for $t<b$) and
\begin{equation}\label{eq:sum_minus}
\sum_{n \in I_-} | t_n - b |^{-\theta} \leq \frac{1}{c} \int_{0}^b |t-b|^{-\theta} \,dt \lesssim_{c,\theta} b^{1-\theta} \lesssim_{\kappa,\theta} a^{1-\theta},
\end{equation}
since $\theta < 1$.
In a similar way, one proves that
\begin{equation}\label{eq:sum_plus}
\sum_{n \in I_+} | t_n - b |^{-\theta} \leq \frac{1}{c} \int_{b}^{\kappa a} |t-b|^{-\theta} \,dt \lesssim_{\kappa,c,\theta} a^{1-\theta}.
\end{equation}
Finally, if $n \in I_0$, then
\[
|b-cn| \leq E + c \lesssim_{\kappa,c,\beta} a^\beta,
\]
which implies that
\begin{equation}\label{eq:num_zero}
\Card I_0 \lesssim_{\kappa,c,\beta} a^{\beta}.
\end{equation}
The estimate \eqref{eq:sum_est} follows by combining \eqref{eq:sum_minus}, \eqref{eq:sum_plus} and \eqref{eq:num_zero}.
\end{proof}

\subsection{A consequence of Lagrange's Mean Value Theorem}
Let $\kappa \geq 1$. Recall from \cite[Definition 6.1]{DM2} the class $\halfpot_1(\kappa)$ of the $C^1$ functions $W : \Rpos \to \Rpos$ such that
\begin{equation}\label{eq:def_halfpot}
\kappa^{-1} W(x) \leq x W'(x) \leq \kappa W(x)
\end{equation}
for all $x \in \Rpos$. In other words, an element of $\halfpot_1(\kappa)$ is ``half of a potential'' in the class $\pot_1(\kappa)$. Indeed, if $V \in \pot_1(\kappa)$, then $V_\oplus,V_\ominus \in \halfpot_1(\kappa)$, where $V_\oplus,V_\ominus : \Rpos \to \Rpos$ are defined by
\begin{equation}\label{eq:Vplusminus}
V_\oplus(x) = V(x), \qquad V_\ominus(x) = V(-x)
\end{equation}
for all $x \in \Rpos$.

We record here some useful properties of functions in the class $\halfpot_1(\kappa)$, including an elementary consequence of Lagrange's Mean Value Theorem, which will be used multiple times later.

\begin{lem}\label{lem:halfpot}
Let $W \in \halfpot_1(\kappa)$.
\begin{enumerate}[label=(\roman*)]
\item\label{en:halfpot_doubling} For all $x \in \Rpos$ and $\lambda \geq 1$,
\[
\lambda^{1/\kappa} W(x) \leq W(\lambda x) \leq \lambda^\kappa W(x).
\]
\item\label{en:halfpot_inverse} $W$ is strictly increasing and invertible, and $W^\inv \in \halfpot_1(\kappa)$ too.
\item\label{en:halfpot_lagrange} For all $x,y \in \Rpos$, if $x \geq y$ then
\[
W(x) - W(y) \simeq_\kappa \frac{W(x)}{x} (x-y) .
\]
\end{enumerate}
\end{lem}
\begin{proof}
Parts \ref{en:halfpot_doubling} and \ref{en:halfpot_inverse} are proved in \cite[Propositions 6.4 and 6.5]{DM2}.

As for part \ref{en:halfpot_lagrange}, if $x \geq 2y$, then $W(x) \geq 2^{1/\kappa} W(y)$ by part \ref{en:halfpot_doubling}, whence
\[
x - y \simeq x, \qquad W(x)-W(y) \simeq_\kappa W(x)
\]
and the desired estimate follows. If instead $x \leq 2y$, then, by Lagrange's Mean Value Theorem,
\[
W(x) - W(y) = W'(\xi) (x-y)
\]
for some $\xi \in (y,x)$, and moreover
\[
W'(\xi) \simeq_\kappa \frac{W(\xi)}{\xi} \simeq_\kappa \frac{W(x)}{x}
\]
by \eqref{eq:def_halfpot} and part \ref{en:halfpot_doubling}, as $x \simeq y \simeq \xi$ in this case, whence the desired estimate again follows.
\end{proof}

\subsection{Bohr--Sommerfeld approximation with logarithmic error}
Let us recall from \cite[Theorem 7.6 and Proposition 7.11]{DM2} some useful estimates involving eigenvalues and sublevel sets of the potential of one-dimensional Schr\"odinger operators.

\begin{prp}\label{prp:BS_rough}
Let $V \in \pot_1(\kappa)$. Then
\begin{equation}\label{eq:bohrsommerfeld}
E_n(V)^{1/2} \, |\{ V \leq E_n(V) \}| \simeq_\kappa n.
\end{equation}
for all $n \in \Npos$. Moreover, for all $E,\lambda \in \Rpos$,
\[
|\{ V \leq \lambda E\}| \simeq_{\kappa,\lambda} |\{ V \leq E\}| \simeq_\kappa |\{ V_\oplus \leq E\}| \simeq_\kappa |\{ V_\ominus \leq E\}|.
\]
\end{prp}

In what follows we will need a sharper version of the estimate \eqref{eq:bohrsommerfeld}.

\begin{prp}\label{prp:BS_log}
Let $V \in \pot_1(\kappa)$. Then, for all $n \in \Npos$,
\[
\left| \int_\RR (E_n(V) - V)_+^{1/2} - \pi n \right| \lesssim_{\kappa} \log(1+n).
\]
\end{prp}

The proof of the above estimate follows the lines of \cite[\S 7.4]{titchmarsh}. In the case $V$ is convex, the logarithmic divergence in the right-hand side can be replaced by a constant, as shown in \cite[\S 7.5]{titchmarsh} and \cite[Theorem 4.2]{dallara_martini}; however the weaker logarithmic bound does not require convexity and will be enough for our purposes.

\begin{proof}[Proof of Proposition \ref{prp:BS_log}]
Let $x_n^\pm \in \Rpos$ be such that $V(\pm x_n^\pm) = E_n \defeq E_n(V)$; in other words, the points $\pm x_n^\pm$ are the transition points corresponding to the energy level $E_n$. Let $y_n^\pm \in (0,x_n^\pm)$ be points to be fixed later, and define $Q_n(x) = (E_n-V(x))^{1/2}$ for $x \in (-x_n^-,x_n^+)$.

By classical Sturm--Liouville theory, $\psi_n \defeq \psi_n(\cdot;V)$ has $n-1$ zeros, which are all in the interval $(-x_n^-,x_n^+)$.
Note now that $V-E_n \geq V(y_n^\pm)-E_n$ on $\pm[y_n^\pm,x_n^\pm]$; by Sturm's comparison theorem, this implies that on $\pm[y_n^\pm,x_n^\pm)$ there are at most
\[
1+(x_n^\pm-y_n^\pm) Q(\pm y_n^\pm)/\pi
\]
zeros of $\psi_n$. Note also that
\[
\int_{\pm[y_n^\pm,x_n^\pm]} (E_n-V)^{1/2} \leq (x_n^\pm-y_n^\pm) Q_n(\pm y_n^\pm).
\]
Hence, if $Z_n^\pm$ denotes the number of zeros of $\psi_n$ in $\pm (0,y_n^\pm)$, then
\begin{multline*}
\left| \int_\RR (E_n - V)_+^{1/2} - \pi n \right| \\\leq 3\pi + 2 \sum_\pm (x_n^\pm-y_n^\pm) Q_n(\pm y_n^\pm)
+ \sum_\pm \left| \int_{\pm (0,y_n^\pm)} (E_n - V)^{1/2} - \pi Z_n^\pm \right|.
\end{multline*}

On the other hand, by \cite[\S 7.3, Lemma]{titchmarsh} (see also \cite[Appendix]{dallara_martini}),
\[
\left| \int_{\pm (0,y_n^\pm)} (E_n - V)^{1/2} - \pi Z_n^\pm \right| \leq \pi + \frac{1}{2} \int_{\pm (0,y_n^\pm)} \frac{|Q_n'|}{Q_n} .
\]
Since $Q'$ is increasing on $(-y_n^\pm,0)$ and decreasing on $(0,y_n^\pm)$,
\[
\int_{\pm (0,y_n^\pm)} \frac{|Q_n'|}{Q_n} = \log \frac{Q_n(0)}{Q_n(\pm y_n^\pm)} = \frac{1}{2} \log \frac{E_n}{E_n-V(\pm y_n^\pm)},
\]
and therefore
\[
\left| \int_\RR (E_n - V)_+^{1/2} - \pi n \right| \leq 5\pi + \sum_\pm \left[ 2 (x_n^\pm-y_n^\pm) Q_n(\pm y_n^\pm)
+ \frac{1}{4} \log \frac{E_n}{E_n-V(\pm y_n^\pm)} \right].
\]

We now choose $y_n^\pm = x_n^\pm - c(x_n^\pm/E_n)^{1/3}$ for a suitable $c>0$. Note that
\[
(x_n^\pm/E_n)^{1/3} = x_n^\pm / (x_n^\pm E_n^{1/2})^{2/3} \simeq_{\kappa} n^{-2/3} x_n^\pm, 
\]
by 
Proposition \ref{prp:BS_rough};
 so, by choosing $c = c(\kappa)$ sufficiently small, we can ensure that $y_n^\pm \simeq_{\kappa} x_n^\pm$. Hence, by 
Lemma \ref{lem:halfpot},
 we deduce that
\[
E_n - V(\pm y_n^\pm) \simeq_{\kappa} \frac{V(\pm x_n^\pm)}{x_n^\pm} (x_n^\pm-y_n^\pm) \simeq_{\kappa} \frac{E_n}{x_n^\pm} \cdot \left(\frac{x_n^\pm}{E_n}\right)^{1/3} = \left(\frac{E_n}{x_n^\pm}\right)^{2/3}
\]
and therefore
\[
(x_n^\pm-y_n^\pm) Q_n(\pm y_n^\pm) \simeq_{\kappa} 1, \quad \frac{E_n}{E_n-V(\pm y_n^\pm)}  \simeq_{\kappa} n^{2/3},
\]
which proves the desired estimate.
\end{proof}

\subsection{A useful change of variables}
The lemma below will be used to extract and exploit the ``rough gap information'' from Proposition \ref{prp:BS_log}.

\begin{lem}\label{lem:BS_function}
For $V \in \pot_1(\kappa)$, define $K_V : \Rpos \to \Rpos$ by
\[
K_V(t) = t^{-1/2} \int_\RR (t-V)_+^{1/2}
\]
for all $t>0$. Then
\[
K_V(t) \simeq_{\kappa} t K_V'(t) \simeq_{\kappa} |\{ V \leq t \}|.
\]
\end{lem}
\begin{proof}
Note first that, if $V_\oplus,V_\ominus$ are defined as in \eqref{eq:Vplusminus}, then
\[
|\{ V \leq t \}| = |\{ V_\oplus \leq t \}| + |\{ V_\ominus \leq t \}| = V_\oplus^\inv(t) + V_\ominus^\inv(t),
\]
and moreover
\[
K_V = K_{V_\oplus} + K_{V_\ominus},
\]
where, for $W : \Rpos \to \Rpos$, we define
\[
K_W(t) = t^{-1/2} \int_0^\infty (t-W)_+^{1/2}.
\]
It is then enough to prove that, for all $W \in \halfpot_1(\kappa)$,
\[
K_W(t) \simeq_{\kappa} t K_W'(t) \simeq_{\kappa} W^{\inv}(t).
\]

Now, for all $t > 0$,
\[
K_W(t) = t^{-1/2} \int_0^{W^{\inv}(t)} (t-W)^{1/2} \simeq_{\kappa} t^{-1/2} \int_0^{W^{\inv}(t)} \frac{x W'(x)}{W(x)} (t-W(x))^{1/2} \,dx ,
\]
and the change of variables $\tau = W(x)/t$ yields
\[
K_W(t) \simeq_{\kappa} \int_0^1 W^{\inv}(\tau t) \, (1-\tau)^{1/2} \frac{d\tau}{\tau} \simeq_{\kappa} W^{\inv}(t);
\]
the last equivalence is consequence of the fact (see Lemma \ref{lem:halfpot}) that $\tau^{\kappa} W^{\inv}(t) \leq W^{\inv}(\tau t) \leq \tau^{1/\kappa} W^{\inv}(t)$ for $\tau \in (0,1)$.

Similarly, one readily sees that
\[
2t K_W'(t) = t^{-1/2} \int_0^{W^{\inv}(t)} \frac{W}{(t-W)^{1/2}} \simeq_{\kappa} t^{-1/2} \int_0^{W^{\inv}(t)} \frac{x W'(x)}{(t-W(x))^{1/2}} \,dx
\]
and again the change of variables $\tau = W(x)/t$ yields
\[
t K_W'(t) \simeq_{\kappa} \int_0^1 W^{\inv}(\tau t) \, (1-\tau)^{-1/2} \,d\tau \simeq_{\kappa} W^{\inv}(t),
\]
as desired.
\end{proof}

\subsection{Proof of the variant of the spectral projector bound}
Here we prove Theorem \ref{thm:new_estimate}, that is, the estimate
\[
\sum_{\substack{n \in \Npos \\ \lambda/\Xi_n(\lambda;V) \in [A,2A] }} \psi_n(x; \Xi_n(\lambda;V) V)^2 \lesssim_{\kappa,a,\theta,\delta} \lambda^{1/2} (\chr_{V \leq 8 A} + e^{-c \lambda^{1/2} |x|} \chr_{V > 8A})
\]
for all $V \in \tilde\pot$, $\lambda,A \in \Rpos$, $x \in \RR$.

Recall from \cite[Theorem 7.7]{DM2} that, for all $V \in \pot_1(\kappa)$, there exists $c=c(\kappa)$ such that
\begin{equation}\label{eq:exp_decay_eigenfunctions}
|\psi_n(x;V)| \lesssim_\kappa |\{ V \leq E_n(V) \}|^{-1/2} \exp(-c |x| \sqrt{V(x)})
\end{equation}
whenever $n \in \Npos$ and $x \in \{ V \geq 4E_n\}$. Recall moreover that, by assumption, $\tilde\pot$ is a subcone of $\pot_1(\kappa)$ such that, for some $\theta,\delta \in (0,1)$ and $a>1$, 
\begin{equation}\label{eq:transition_eigenfunctions}
|\psi_n(x;V)| \leq a \, |\{ V \leq E_n(V) \}|^{-1/2}  \min\{ n^{\delta/2} , E_n(V)^{\theta/2} |V(x)-E_n(V)|^{-\theta/2} \} 
\end{equation}
for all $V \in \tilde\pot$, $x \in \RR$, $n \in \Npos$. For simplicity, in the rest of the proof, we will write $\lesssim$ and $\simeq$ instead of $\lesssim_{\kappa,a,\theta,\delta}$ and $\simeq_{\kappa,a,\theta,\delta}$.

Fix $V \in \tilde\pot$ and let $\Xi_n \defeq \Xi_n(\cdot;V)$.
First note that, if $V(x) > 8A$ and $\lambda/\Xi_n(\lambda) \in [A,2A]$, then $\Xi_n(\lambda) V(x) > 4 \lambda$, and therefore by \eqref{eq:exp_decay_eigenfunctions} we deduce that
\[\begin{split}
\psi_n(x;\Xi_n(\lambda) V)^2 &\lesssim |\{ V \leq \lambda/\Xi_n(\lambda) \}|^{-1} \exp(-2c |x| \sqrt{\Xi_n(\lambda) V(x)}) \\
&\leq |\{ V \leq A \}|^{-1} \exp(-4c\lambda^{1/2} |x|).
\end{split}\]
On the other hand, by Proposition \ref{prp:BS_rough},
\begin{equation}\label{eq:roughBS}
\lambda^{1/2} |\{ V \leq A \}| \simeq \lambda^{1/2} |\{ V \leq \lambda/\Xi_n(\lambda) \}| \simeq n
\end{equation}
 so the number of summands is $\lesssim \lambda^{1/2} |\{ V \leq A \}|$, and we deduce that
\[
\sum_{\substack{n \in \Npos \\ \lambda/\Xi_n(\lambda;V) \in [A,2A] }}  \psi_n(x;\Xi_n(\lambda) V)^2 \lesssim \lambda^{1/2} \exp(-4c\lambda^{1/2} |x|)
\]
whenever $V(x) > 8A$.

It remains to prove the uniform bound on $\{V \leq 8A\}$. For this, we use \eqref{eq:transition_eigenfunctions} to obtain that
\begin{equation}\label{eq:pointwise_rewritten}
\begin{split}
\psi_n(x;\Xi_n(\lambda) V)^2 &\lesssim |\{ V \leq \lambda/\Xi_n(\lambda) \}|^{-1}  \min\{ n^{\delta} , \lambda^{\theta} |\Xi_n(\lambda) V(x)-\lambda|^{-\theta} \} \\
&\simeq |\{ V \leq A \}|^{-1} \min\{ n^{\delta} , A^{\theta} |V(x)-\lambda/\Xi_n(\lambda)|^{-\theta} \}.
\end{split}
\end{equation}
Define $K_V$ as in Lemma \ref{lem:BS_function}, and recall that
\begin{equation}\label{eq:K_V_prop}
K_V(t) \simeq t K_V'(t) \simeq |\{ V \leq t \}|.
\end{equation}
Since $\lambda/\Xi_n(\lambda) \simeq A \gtrsim V(x)$, we deduce, by Lemma \ref{lem:halfpot},
 that
\[\begin{split}
|V(x)-\lambda/\Xi_n(\lambda)| &\simeq \frac{A}{K_V(A)} |K_V(V(x)) - K_V(\lambda/\Xi_n(\lambda))| \\
&= \frac{A}{\lambda^{1/2} K_V(A)} |\lambda^{1/2} K_V(V(x)) - \lambda^{1/2} K_V(\lambda/\Xi_n(\lambda))|.
\end{split}\]
Set now $a = \lambda^{1/2} K_V(A)$, $b = \lambda^{1/2} K_V(V(x))$, $t_n = \lambda^{1/2} K_V(\lambda/\Xi_n(\lambda))$, and observe that $t_n \simeq a \simeq n \gtrsim b$ by \eqref{eq:roughBS} and \eqref{eq:K_V_prop}, so the bound \eqref{eq:pointwise_rewritten} can be rewritten as
\[
\psi_n(x;\Xi_n(\lambda) V)^2 \lesssim \lambda^{1/2}  \min\{ a^{\delta-1} , a^{\theta-1} |b-t_n|^{-\theta} \}.
\]
Furthermore,
\[
t_n = \lambda^{1/2} K_V(\lambda/\Xi_n(\lambda)) = \int_\RR (\lambda-\Xi_n(\lambda) V)_+^{1/2},
\]
and therefore
\[
|t_n - \pi n| \lesssim \log(1+n) \lesssim n^{1-\delta}
\]
by Proposition \ref{prp:BS_log} (applied to the potential $\Xi_n(\lambda) V$) and the fact that $\delta<1$. As a consequence, we can apply Lemma \ref{lem:sum_int} and obtain that
\[
\sum_{\substack{n \in \Npos \\ \lambda/\Xi_n(\lambda;V) \in [A,2A] }} \psi_n(x;\Xi_n(\lambda) V)^2 \lesssim \lambda^{1/2} \sum_{\substack{n \in \Npos \\ t_n \simeq a }}  \min\{ a^{\delta-1} , a^{\theta-1} |b-t_n|^{-\theta} \} \lesssim \lambda^{1/2},
\]
as desired.

\section{Pointwise eigenfunction estimates in the transition region}\label{s:pointwiseeigenfunction}

\subsection{Summary of the results}
Let $\kappa \geq 1$. Let us introduce the following subclasses of $\pot_1(\kappa)$. Recall that a \emph{modulus of continuity} is a function $\omega : [0,\infty] \to [0,\infty]$ such that $\lim_{t \to 0} \omega(t) = 0$.

\begin{dfn}
If $\omega$ is a modulus of continuity, let $\pot_{1,\uc}(\kappa,\omega)$ be the class of potentials $V \in \pot_1(\kappa)$ such that
\[
\left|\log (V'(\pm e^{t}) /V'(\pm e^{t'}))\right| \leq \omega(|t-t'|) \qquad\text{for all } t,t'\in \RR.
\]
In other words, $\omega$ is a modulus of continuity for the functions $t \mapsto \log  |V'(\pm e^t)|$.
\end{dfn}

\begin{rem}\label{rem:hoelder_uc}
It is easy to see that, for all $\theta \in (0,1)$, $\pot_{1+\theta}(\kappa) \subseteq \pot_{1,\uc}(\kappa,\omega_{\kappa,\theta})$, where $\omega_{\kappa,\theta}$ is a suitable modulus of continuity such that $\omega_{\kappa,\theta}(t) \simeq_{\kappa,\theta} t^\theta$ for $t$ small.
\end{rem}

\begin{dfn}
Let $\pot_{1,\cv}(\kappa)$ be the class of the convex potentials in $\pot_1(\kappa)$.
\end{dfn}

\begin{dfn}
For $k \geq 2$, let $\pot_k(\kappa)$ be the class of the potentials $V \in \pot_1(\kappa)$ which are $C^k$ on $\RR \setminus \{0\}$ and satisfy the estimates
\[
|x^{\ell} V^{(\ell)}(x)| \leq \kappa V(x) \qquad \text{for all } x \neq 0 \text{ and } \ell=2,\dots,k.
\]
\end{dfn}

The aim of this section is to prove the following pointwise estimates for the eigenfunctions of $\opH[V] = -\partial_x^2 + V$.

\begin{thm}\label{thm:pointwise_est}
For all $x \in \RR$ and $n \in \Npos$, the estimates
\begin{align*}
|\psi_n(x;V)| &\lesssim_{\bar\kappa,\alpha} \frac{1}{|\{V \leq E_n(V)\}|^{1/2}} \min \{ n^{2\alpha/3}, |1-V(x)/E_n(V)|^{-\alpha} \} , \\
|\psi_n'(x;V)| &\lesssim_{\bar\kappa,\alpha} \frac{E_n(V)^{1/2}}{|\{V \leq E_n(V)\}|^{1/2}} \max \{ n^{(2\alpha-1)/3}, (1-V(x)/E_n(V))_+^{1/2-\alpha} \} 
\end{align*}
hold in the following cases:
\begin{enumerate}[label=(\roman*)]
\item\label{en:pointwise_est_weak} with $\bar\kappa=\kappa$ and $\alpha=1/2$, whenever $V \in \pot_1(\kappa)$;
\item\label{en:pointwise_est_new} with $\bar\kappa=(\kappa,\omega)$ and $\alpha \in (1/4,1/2)$, whenever $V \in \pot_{1,\uc}(\kappa,\omega)$;
\item\label{en:pointwise_est_conv} with $\bar\kappa=\kappa$ and $\alpha = 1/4$, whenever $V \in \pot_{1,\cv}(\kappa)$;
\item\label{en:pointwise_est_c3} with $\bar\kappa=\kappa$ and $\alpha = 1/4$, whenever $V \in \pot_3(\kappa)$.
\end{enumerate}
\end{thm}

We point out that, well inside the classical region (say, where $V(x) \leq E_n(V)/2$), the above bounds reduce to the uniform bound stated, e.g., in \cite[Proposition 6.2]{DM2}, which applies to any $V \in \pot_1(\kappa)$; similarly, far from the classical region (say, where $V(x) \geq 4E_n(V)$), a much better (exponentially decaying) bound is known to hold, again for arbitrary $V \in \pot_1(\kappa)$ (see, e.g., \cite[Theorem 7.7]{DM2}).  As anticipated in the introduction, the relevance of the above bounds is therefore their validity in the transition region, where $V(x) \simeq E_n(V)$.

We also point out that the bound for $\psi_n$ in Theorem \ref{thm:pointwise_est}\ref{en:pointwise_est_c3} matches the one obtained in \cite[Proposition 3.4]{dallara_martini} under the additional assumption $V(x) \simeq |x|^d$ for some $d>1$. The method used in \cite{dallara_martini} is based on a theorem by Olver \cite{olverbook}, which essentially allows one to approximate $\psi_n$ with a suitably rescaled Airy function, so the bounds for $\psi_n$ can be reduced to known bounds for the Airy function. The method presented here,
instead, does not go through such an approximation, but yields the desired bounds directly. Moreover it
allows us to treat potentials $V(x) \simeq |x|^d$ for arbitrary $d>0$, or even potentials that are not comparable to a single power law (provided they belong to one of the classes of potentials defined above); see also the discussion in the introduction of \cite{DM2}.

Here and in the following sections, we shall write $E_n$ and $\psi_n$ in place of $E_n(V)$ and $\psi_n(\cdot;V)$ when the potential $V$ is clear from the context.

\subsection{Pointwise estimate in the classical region: \texorpdfstring{$C^1$}{C1} potentials}\label{ss:c1_estimate}

Here we assume that $V \in \pot_1(\kappa)$ and prove the pointwise estimate
\begin{equation}\label{eq:eigen_pointwise_weak}
\psi_n^2 + \frac{(\psi_n')^2}{E_n-V} \lesssim_{\kappa} \frac{E_n (E_n-V)^{-1}}{|\{V \leq E_n\}|} 
\end{equation}
in the classical region $\{V < E_n \}$.

The key to the proof is the monotonicity information provided by the following
elementary identity,
 valid on $\RR \setminus \{0\}$:
\begin{equation}\label{eq:monot_num}
\left( (E_n-V) \psi_n^2 + (\psi_n')^2 \right)' = - V' \psi_n^2
\end{equation}
(cf.\ \cite[Proposition 5.3]{DM2}). As the right-hand side is positive on $(-\infty,0)$ and negative on $(0,\infty)$, we conclude that
\begin{equation}\label{eq:monot_zero}
(E_n-V) \psi_n^2 + (\psi_n')^2 \leq E_n \psi_n(0)^2 + \psi_n'(0)^2
\end{equation}
on the whole $\RR$.

To bound the right-hand side of the latter, we use another monotonicity argument, based on the following counterpart to \eqref{eq:monot_num}:
\begin{equation}\label{eq:monot_den}
\left( \psi_n^2 + \frac{(\psi_n')^2}{E_n-V} \right)' = \frac{V'}{(E_n-V)^2}  (\psi_n')^2
\end{equation}
on the region $\{ V \neq E_n \} \setminus \{0\}$ (cf.\ \cite[\S8.3]{titchmarsh}). As the right-hand side is negative on $(-\infty,0)$ and positive on $(0,\infty)$, we can control (cf.\ \cite[Section 6.4]{DM2}) the value of $\psi_n^2 + \frac{(\psi_n')^2}{E_n-V}$ at $0$ with its average on $\{V \leq E_n/2\}$, thus obtaining that
\begin{equation}\label{eq:zero_average}
\psi_n(0)^2 + \frac{\psi_n'(0)^2}{E_n} \leq \frac{1}{|\{V \leq E_n/2\}|} \left( \|\psi_n\|_2^2 + \frac{2}{E_n} \|\psi_n'\|_2^2 \right) \lesssim_{\kappa} \frac{1}{|\{V \leq E_n\}|}
\end{equation}
by Proposition \ref{prp:BS_rough}.

The desired estimate \eqref{eq:eigen_pointwise_weak} then follows by combining \eqref{eq:monot_zero} and \eqref{eq:zero_average}.

We record here an elementary consequence of \eqref{eq:eigen_pointwise_weak}, that is, a uniform estimate which is valid well within the classical region: namely, for any $\theta \in (0,1)$,
\begin{equation}\label{eq:uniform_wellwithin}
\psi_n^2 + \frac{(\psi_n')^2}{E_n-V} \lesssim_{\kappa,\theta} \frac{1}{|\{V \leq E_n\}|} \qquad\text{on } \{V \leq (1-\theta) E_n\}. 
\end{equation}

\subsection{Pointwise estimate in the classical region: \texorpdfstring{$C^3$}{C3} potentials}\label{ss:sonin}

In this section we assume that $V \in \pot_3(\kappa)$, and prove the following ``improvement'' of \eqref{eq:eigen_pointwise_weak}:
\begin{equation}\label{eq:est_eigen_pointwise}
\psi_n^2 + \left(1 + \frac{n^{-2/3} E_n}{E_n-V}\right)^{-3} \frac{(\psi_n')^2}{E_n-V} \lesssim_{\kappa} \frac{E_n^{1/2} (E_n-V)^{-1/2}}{|\{V \leq E_n\}|} 
\end{equation}
in the classical region $\{V < E_n\}$.

As in Section \ref{ss:c1_estimate} above, the proof is based on a monotonicity argument. Specifically, the method used here is inspired by what is referred to as the ``Sonin's function'' method in \cite{krasikov}. The main idea is to consider the function
\[
f_n \defeq (E_n - V)^{1/4} \psi_n,
\]
which is well defined and $C^2$ in the punctured classical region $\{ V < E_n\} \setminus \{0\}$. From the differential equation
\begin{equation}\label{eq:ode}
-\psi_n'' + V \psi_n = E_n \psi_n
\end{equation}
satisfied by $\psi_n$, one readily obtains that
\begin{equation}\label{eq:new_ode}
f_n'' - 2A_n f_n' + B_n f_n = 0
\end{equation}
in the punctured classical region, where
\[
A_n = - \frac{1}{4} \frac{V'}{E_n-V}, \qquad B_n = E_n-V + \frac{5}{16} \frac{(V')^2}{(E_n-V)^2}+ \frac{1}{4} \frac{V''}{E_n-V}.
\]
We now consider the ``Sonin's function'' for $f_n$, namely,
\[
S_n = f_n^2 + \frac{(f_n')^2}{B_n},
\]
which is defined and $C^1$ in the subset $\{ V < E_n, \, B_n \neq 0\}$ of the punctured classical region.
A quick computation shows that
\begin{equation}\label{eq:sonin_derivative}
S_n' = \frac{4A_n B_n-B_n'}{B_n^2} (f_n')^2,
\end{equation}
that is, the derivative of $S_n$ has the same sign as
\[4A_n B_n - B_n' 
= - \frac{15}{16} \frac{(V')^3}{(E_n-V)^3} - \frac{9}{8} \frac{V' V''}{(E_n-V)^2} - \frac{1}{4} \frac{V'''}{E_n-V}.
\]

To study the sign of $B_n$ and $4A_n B_n - B_n'$, we rewrite them as
\begin{align*}
B_n &= E_n-V + \frac{5}{16} \frac{(V')^2}{(E_n-V)^2} \left[ 1 + \frac{4}{5} \frac{V''}{(V')^2} (E_n-V) \right],\\
4A_n B_n - B_n' &= - \frac{15}{16} \frac{(V')^3}{(E_n-V)^3} \left[ 1 + \frac{6}{5} \frac{V''}{(V')^2} (E_n-V) + \frac{4}{15} \frac{V'''}{(V')^3} (E_n-V)^2 \right] .
\end{align*}
This is convenient because, for $V \in \pot_3(\kappa)$,
\begin{equation}\label{eq:pot_trider_estimate}
|V''/(V')^2| \lesssim_{\kappa} 1/V, \qquad |V'''/(V')^3| \lesssim_{\kappa} 1/V^2.
\end{equation}
Hence, we can choose $\epsilon = \epsilon(\kappa) > 0$ sufficiently small to guarantee that
\begin{equation}\label{eq:bn_est}
 B_n \simeq E_n-V + \frac{(V')^2}{(E_n-V)^2}, \qquad 4A_n B_n - B_n' \simeq - \frac{(V')^3}{(E_n-V)^3}
\end{equation}
in the region $\Omega_n = \{ (1-\epsilon) E_n \leq V < E_n \}$. In particular, $B_n>0$ there, and $S_n'$ has the same sign as $-V'$.

Define now $y_n^\pm > 0$ as the points such that $V(\pm y_n^\pm) = (1-\epsilon) E_n$. Then, using the fact that $S_n$ is decreasing on $\Omega_n \cap (0,\infty)$ and increasing on $\Omega_n \cap (-\infty,0)$, we conclude that
\begin{equation}\label{eq:sonin_est}
f_n^2 + \frac{(f_n')^2}{B_n} \leq \max_{\pm} \left( f_n(\pm y_n^\pm)^2 + \frac{f_n'(\pm y_\pm)^2}{B_n(\pm y_n^\pm)} \right) \qquad\text{on $\Omega_n$}.
\end{equation}
Note now that
\[
f_n = (E_n-V)^{1/4} \psi_n, \qquad f_n' = (E_n-V)^{1/4}  \left[ \psi_n'-\frac{1}{4} \frac{V'}{E_n-V} \psi_n \right] .
\]
In particular, by \eqref{eq:bn_est},
\begin{equation}\label{eq:fpsiequiv}
f_n^2 +\frac{(f_n')^2}{B_n} \simeq (E_n-V)^{1/2} \left[ \psi_n^2  + \frac{(\psi_n')^2}{B_n}\right]
\end{equation}
in the region $\Omega_n$, and
\[
f_n(\pm y_n^\pm)^2 +\frac{f_n'(y_n^\pm)^2}{B_n(y_n^\pm)} \lesssim (\epsilon E_n)^{1/2} \left[ \psi_n(\pm y_n^\pm)^2 + \frac{\psi_n'(\pm y_n^\pm)^2}{\epsilon E_n} \right] 
\lesssim_{\kappa} \frac{E_n^{1/2}}{|\{V \leq E_n\}|};
\]
the last inequality is consequence of the fact that $\pm y_n^\pm \in \{ V \leq (1-\epsilon) E_n\}$, that is, $\pm y_n^\pm$ are well inside the classical region, so the uniform estimate \eqref{eq:uniform_wellwithin}
applies. From \eqref{eq:sonin_est} and \eqref{eq:fpsiequiv} we then deduce that
\begin{equation}\label{eq:est_eigen_pointwise_var}
\psi_n^2 + \frac{(\psi_n')^2}{B_n} \lesssim_{\kappa} \frac{E_n^{1/2} (E_n-V)^{-1/2}}{|\{V \leq E_n\}|} \qquad\text{on } \Omega_n.
\end{equation}

As $|V'| \simeq_\kappa |\{ V \leq E_n \}|^{-1} E_n \simeq_\kappa n^{-1} E_n^{3/2}$ on $\Omega_n$ by Proposition \ref{prp:BS_rough},
from \eqref{eq:bn_est} we deduce that
\[
B_n \simeq E_n-V + \frac{(V')^2}{(E_n-V)^2} \simeq_\kappa (E_n-V) \left(1+ \frac{n^{-2/3} E_n}{E_n-V}\right)^3,
\]
so from \eqref{eq:est_eigen_pointwise_var} the desired estimate \eqref{eq:est_eigen_pointwise} follows on $\Omega_n$. Actually the same estimate \eqref{eq:est_eigen_pointwise} holds on the whole classical region $\{ V < E_n \}$, because on $\{ V \leq (1-\epsilon) E_n\}$ it simply reduces to the uniform estimate \eqref{eq:uniform_wellwithin}.

\subsection{Pointwise estimate in the classical region: convex potentials and potentials with uniformly continuous derivative}\label{ss:sonin2}

A variation of the method presented in the previous section allows us to show that, if $V \in \pot_{1,\uc}(\kappa,\omega)$, then the following pointwise estimate holds for all $\alpha \in (1/4,1/2)$:
\begin{equation}\label{eq:est_eigen_pointwise_new}
\psi_n^2 + \left(1 + \frac{n^{-2/3} E_n}{E_n-V}\right)^{-3} \frac{(\psi_n')^2}{E_n-V} \lesssim_{\bar\kappa,\alpha} \frac{E_n^{2\alpha} (E_n-V)^{-2\alpha}}{|\{V \leq E_n\}|}
\end{equation}
in the classical region $\{V < E_n\}$, where $\bar\kappa = (\kappa,\omega)$; 
in addition, we will prove the analogous estimate for $\alpha=1/4$ and $\bar\kappa = \kappa$ in the case $V \in \pot_{1,\cv}(\kappa)$.

As in Section \ref{ss:sonin}, we apply the Sonin's function method. We only discuss the estimates for $x \geq 0$, as the case $x \leq 0$ can be be treated analogously. Let $x_n \in \Rpos$ denote the positive transition point, i.e., $V(x_n) = E_n$.

Let $\alpha \in \Rpos$, and consider the function
\[
f_n(x) \defeq (x_n - x)^{\alpha} \psi_n(x),
\]
which is well defined and $C^2$ in the classical region. From the differential equation \eqref{eq:ode}
satisfied by $\psi_n$, one readily obtains that $f_n$ satisfies \eqref{eq:new_ode}
in the classical region, where
\begin{equation}\label{eq:new_an_bn}
A_n = - \alpha (x_n-x)^{-1}, \qquad B_n = E_n-V + \alpha(\alpha+1) (x_n-x)^{-2} .
\end{equation}
Since $\alpha>0$, clearly $B_n > 0$ in the classical region.
We now consider the Sonin's function for $f_n$, namely,
\[
S_n = f_n^2 + \frac{(f_n')^2}{B_n},
\]
which is defined and $C^1$ in the punctured classical region.
By arguing as in \eqref{eq:sonin_derivative}, 
we deduce that
the derivative of $S_n$ has the same sign as
\[4A_n B_n - B_n' 
= V'-4\alpha\frac{E_n-V}{x_n-x} - 2\alpha(\alpha+1)(2\alpha+1) (x_n-x)^{-3}.
\]

To study the sign of $4A_n B_n - B_n'$, we observe that, by Lagrange's Mean Value Theorem,
\begin{equation}\label{eq:monotonicity_alpha}
V'(x)-4\alpha\frac{E_n-V(x)}{x_n-x} = V'(x) - 4\alpha V'(\xi) = -V'(x) (4\alpha V'(\xi)/V'(x) - 1)
\end{equation}
for some $\xi \in (x,x_n)$.

Now, if $V \in \pot_{1,\uc}(\kappa,\omega)$, then
\[
\left|\log (V'(\xi)/V'(x)) \right| \leq \omega( \log(\xi/x) )
\]
and therefore
\[
4\alpha V'(\xi)/V'(x) = \exp(\log(4\alpha) + \log(V'(\xi)/V'(x))) \geq \exp(\log(4\alpha) - \omega(\log(\xi/x))).
\]
If we take $\alpha \in (1/4,1/2)$, then $\log(4\alpha)>0$. Since $\lim_{t \to 0} \omega(t) = 0$, we can find $\delta = \delta(\alpha,\omega)>0$ such that
\[
\omega(t) \leq \log(4\alpha) \text{ whenever } 0 < t \leq \delta.
\]
Consequently, whenever $0 < x_n/x \leq e^\delta$, we have $4\alpha V'(\xi)-V'(x) \geq 0$, and therefore $4A_n B_n -B_n' < 0$. As a consequence, $S_n$ is decreasing in the interval $[e^{-\delta} x_n,x_n)$.

If instead $V \in \pot_{1,\cv}(\kappa)$, then an even simpler argument applies. Indeed, one can go back to \eqref{eq:monotonicity_alpha}, and observe that
$V'(\xi)/V'(x) \geq 1$ in this case, because $V'$ is increasing; consequently we obtain that $4A_n B_n - B'_n < 0$ on the whole $\Rpos \cap \{ V < E_n\}$ whenever $\alpha \geq 1/4$. In particular, we can take $\alpha=1/4$ and $\delta=1$ in this case, and again conclude that $S_n$ is decreasing in the interval $[e^{-\delta} x_n,x_n)$.

The fact that $S_n$ is decreasing on $[e^{-\delta} x_n,x_n)$ implies that
\begin{equation}\label{eq:sonin_est_new}
f_n^2 + \frac{(f_n')^2}{B_n} \leq  f_n(e^{-\delta} x_n)^2 + \frac{f_n'(e^{-\delta} x_n)^2}{B_n(e^{-\delta} x_n)}  \qquad\text{on $[e^{-\delta}x_n,x_n)$}.
\end{equation}
Note now that
\[
f_n = (x_n-x)^{\alpha} \psi_n, \qquad f_n' = (x_n-x)^{\alpha}  \left[ \psi_n'-\alpha (x_n-x)^{-1} \psi_n \right] .
\]
In particular, by \eqref{eq:new_an_bn},
\begin{equation}\label{eq:fpsiequiv_new}
f_n^2 +\frac{(f_n')^2}{B_n} \simeq_\alpha (x_n-x)^{2\alpha} \left[ \psi_n^2  + \frac{(\psi_n')^2}{B_n}\right]
\end{equation}
on $[e^{-\delta}x_n,x_n)$, and
\[\begin{split}
f_n(e^{-\delta} x_n)^2 +\frac{f_n'(e^{-\delta} x_n)^2}{B_n(e^{-\delta} x_n)} &\lesssim_{\bar\kappa,\alpha} x_n^{2\alpha} \left[ \psi_n(e^{-\delta} x_n)^2 + \frac{\psi_n'(e^{-\delta} x_n)^2}{E_n} \right] \\
&\lesssim_{\bar\kappa,\alpha} \frac{x_n^{2\alpha}}{|\{V \leq E_n\}|};
\end{split}\]
these inequalities are consequence of the fact that $e^{-\delta} x_n$ is well within the classical region (see Lemma \ref{lem:halfpot}\ref{en:halfpot_doubling}), so $E_n - V(e^{-\delta} x_n) \simeq_{\bar\kappa,\alpha} E_n$ and the uniform estimate \eqref{eq:uniform_wellwithin}
applies. From \eqref{eq:sonin_est_new} and \eqref{eq:fpsiequiv_new} we then deduce that
\begin{equation}\label{eq:est_eigen_pointwise_var_new}
\psi_n^2 +\frac{(\psi_n')^2}{B_n} \lesssim_{\bar\kappa,\alpha} \frac{x_n^{2\alpha} (x_n-x)^{-2\alpha}}{|\{V \leq E_n\}|} \quad\text{on } [e^{-\delta} x_n,x_n).
\end{equation}

We now observe that, by Proposition \ref{prp:BS_rough},
\[
x_n \simeq_{\kappa} |\{ V \leq E_n\}| 
\]
and, by Lemma \ref{lem:halfpot},
\[
E_n - V(x) \simeq_{\kappa} \frac{E_n}{|\{V \leq E_n\}|} (x_n-x) 
\]
on $[e^{-\delta} x_n,x_n)$; hence, from \eqref{eq:new_an_bn} and Proposition \ref{prp:BS_rough} we deduce that
\[
B_n \simeq_{\kappa} (E_n - V) \left(1 + \frac{n^{-2/3} E_n}{ (E_n-V)^2} \right)^3
\]
on $[e^{-\delta} x_n,x_n)$.
As a consequence, the estimate \eqref{eq:est_eigen_pointwise_var_new} gives \eqref{eq:est_eigen_pointwise_new} on $[e^{-\delta}x_n,x_n)$. On the other hand, the interval $[0,e^{-\delta} x_n)$ is well within the classical region, so on that interval the estimate \eqref{eq:est_eigen_pointwise_new} follows from the uniform estimate \eqref{eq:uniform_wellwithin}. In conclusion, \eqref{eq:est_eigen_pointwise_new} is proved on $[0,x_n) = \{ V < E_n \} \cap [0,\infty)$, as desired.

\subsection{Zeros and local extrema of eigenfunctions and their derivatives}\label{ss:location}

In this section we prove the estimates for $\psi_n'$ of Theorem \ref{thm:pointwise_est} on the whole $\RR$, as well as the corresponding estimates for $\psi_n$ within the classical region. These estimates will be derived from those proved in the previous sections, combined with information on the location of the extremum points of $\psi_n$ and $\psi_n'$.

Assume at first that $V \in \pot_1(\kappa)$. We recall a few basic facts about zeros and local extrema of $\psi_n$ and $\psi_n'$, which are easy consequences of the fact that $\psi_n$ is a square-integrable solution of \eqref{eq:ode}.
\begin{enumerate}[label=(\alph*)]
\item $\psi_n$ and $\psi_n'$ do not vanish simultaneously at any point.
\item\label{en:monot_outside} The zeros of $\psi_n$ and $\psi_n'$ are contained in the classical region $\{V < E_n\}$; outside of the classical region, $x \psi_n(x) \psi_n'(x) < 0$, which implies that both $\psi_n^2$ and $(\psi_n')^2$ are strictly increasing on $\{V \geq E_n\} \cap (-\infty,0)$ and strictly decreasing on $\{V \geq E_n\} \cap (0,\infty)$.
\item $\psi_n$ has $n-1$ zeros in the classical region, which are all simple, so $\psi_n$ changes sign at each zero.
\item The zeros of $\psi_n''$ are the zeros of $\psi_n$ and the two transition points (i.e., the points where $V_n = E$); these are the inflexion points of $\psi_n$.
\item Between two consecutive zeros of $\psi_n''$, the function $\psi_n$ is strictly concave or strictly convex according to whether $\psi_n$ is positive or negative.
\item Between two consecutive zeros of $\psi_n''$, there is exactly one zero of $\psi_n'$; these are all the zeros of $\psi_n'$, which has $n$ zeros, and they are all simple.
\item Similarly, between two consecutive zeros of $\psi_n'$, there is exactly one zero of $\psi_n$.
\item The zeros of $\psi_n'$ are the local maximum/minimum points of $\psi_n$, that is, the local maximum points of $\psi_n^2$.
\item Similarly, the zeros of $\psi_n''$ are the local maximum/minimum points of $\psi'$, that is, the local maximum points of $(\psi_n')^2$.
\end{enumerate}

Since $|\psi'|$ attains its maximum in the classical region $\{ V<E_n\}$, from \eqref{eq:eigen_pointwise_weak} we derive the following uniform estimate for $\psi_n'$:
\begin{equation}\label{eq:uniform_der}
\| \psi_n' \|_\infty = \sup_{\{ V < E_n \} } |\psi_n'| \lesssim_\kappa \frac{E_n^{1/2}}{|\{V \leq E_n\}|^{1/2}}.
\end{equation}
This proves the estimate for $\psi_n'$ in Theorem \ref{thm:pointwise_est}\ref{en:pointwise_est_weak}, and moreover implies a rough uniform estimate for $\psi_n$:
\begin{equation}\label{eq:uniform_psi_rough}
\| \psi_n \|_\infty = \sup_{\{ V < E_n \} } |\psi_n| \leq |\psi_n(0)| + \int_{\{V < E_n \}} |\psi_n'| \lesssim_\kappa \frac{n}{|\{V \leq E_n\}|^{1/2}},
\end{equation}
where \eqref{eq:zero_average} and Proposition \ref{prp:BS_rough} were used. Notice that the estimate \eqref{eq:uniform_psi_rough} is worse than the uniform estimate for $\psi_n$ in Theorem \ref{thm:pointwise_est}\ref{en:pointwise_est_weak}; to prove the latter, a more careful analysis of the local extrema of $\psi_n$ and $\psi_n'$ is needed.

Further important information about local extrema is deduced from the monotonicity identities \eqref{eq:monot_den} and \eqref{eq:monot_num}. These identities give us precise information on the sign of the derivatives of the two functions $(E_n-V) \psi_n^2 + (\psi_n')^2$ and $\psi_n^2 + \frac{(\psi_n')^2}{E_n-V}$; by evaluation at the zeros of $\psi_n'' = (V-E_n) \psi_n$ and $\psi_n'$, these yield the following information.
\begin{enumerate}[label=(\alph*),resume]
\item The local maxima of $\psi_n^2$ (that is, the values of $\psi_n^2$ at the zeros of $\psi_n'$) on $[0,\infty)$ form a strictly increasing sequence, while on $(-\infty,0]$ they form a strictly decreasing sequence.
\item\label{en:monot_dermax} The local maxima of $(\psi_n')^2$ (that is, the values of $(\psi_n')^2$ at the zeros of $\psi_n''$) on $[0,\infty)$ form a strictly decreasing sequence, while on $(-\infty,0]$ they form a strictly increasing sequence.
\end{enumerate}
In particular, the global maximum of $\psi_n^2$ is attained at an \emph{outermost} zero of $\psi_n'$ (that is, a zero closest to one of the two transition points). Similarly, for $n>1$, the global maximum of $(\psi_n')^2$ is attained at an \emph{innermost} zero of $\psi_n$ (that is, the origin if $\psi_n(0)=0$, or the positive and negative zeros of $\psi_n$ that are closest to the origin if $\psi_n(0) \neq 0$). For this reason, it is useful to investigate the location of the zeros of $\psi_n$ and $\psi_n'$ within the classical region.

To this purpose, as in \cite[\S 6.31]{szego}, we can fruitfully use Sturm's comparison theorem. Namely, for any $\tilde E \in (0,E_n)$, we have that $V-E_n \leq \tilde E - E_n$ on $\{V \leq \tilde E\}$. Hence, we can find a zero of $\psi_n$ between any two zeros of a nontrivial solution of $-u'' + (\tilde E-E_n) u = 0$ on $\{V \leq \tilde E\}$; in other words, we have proved the following.
\begin{enumerate}[label=(\alph*),resume]
\item For all $\tilde E<E_n$, there is a zero of $\psi_n$ in any interval of length $\pi/\sqrt{E_n-\tilde E}$ fully contained in $\{V \leq \tilde E\}$.
\end{enumerate}
In order to be able to apply this result, we need to ensure that such an interval exists, that is, we need to choose $\tilde E$ so that
\[
 |\{V \leq \tilde E\}| \sqrt{ E_n-\tilde E} \geq \pi.
\]

We now observe that, by Proposition \ref{prp:BS_rough},
\begin{gather*}
\left(\frac{E_n}{n^{2/3}}\right)^{1/2} |\{ V \leq E_n/2 \}|  \simeq_{\kappa} n^{2/3}.
\end{gather*}
This means that there exists $n_0 = n_0(\kappa) \in \Npos$ sufficiently large that
\[
\frac{E_n}{n^{2/3}} \leq \frac{ E_n }{2}, \quad \left( \frac{E_n}{n^{2/3}} \right)^{1/2} |\{ V \leq E_n/2 \}| \geq  \pi \quad\text{ for all } n \geq n_0.
\]
Consequently, if we take $\tilde E = E_n - \frac{E_n}{n^{2/3}} $, then
\[
|\{V \leq \tilde E\}| \sqrt{E_n-\tilde E} \geq  \left( \frac{E_n}{n^{2/3}} \right)^{1/2} |\{ V \leq E_n/2 \}| \geq \pi \quad\text{ for all } n \geq n_0,
\]
and the previous result can be applied.

We now observe that, if $x_n^\pm, y_n^\pm \in (0,\infty)$ are such that $V(\pm x_n^\pm) = E_n$ and $V(\pm y_n^\pm) = \tilde E$, then, by Lemma \ref{lem:halfpot} and Proposition \ref{prp:BS_rough},
\[
x_n^\pm - y_n^\pm \simeq_{\kappa} \frac{x_n^\pm}{E_n} (E_n - \tilde E) \simeq_{\kappa} \frac{|\{V \leq E_n\}|}{n^{2/3}} \simeq_{\kappa} \frac{n^{1/3}}{E_n^{1/2}} \simeq \frac{\pi}{\sqrt{E_n-\tilde E}}.
\]

In conclusion, for all $n \geq n_0 = n_0(\kappa)$:
\begin{enumerate}[label=(\alph*),resume]
\item\label{en:dist_zero_tran} There are $\gtrsim_{\kappa} n^{2/3}$  zeros of $\psi_n$ within the region $\{ V \leq E_n - E_n/n^{2/3} \}$, and the outermost of them have distance $\simeq_{\kappa} n^{-2/3} |\{V \leq E_n\}|$ from the transition point with the same sign.
\item\label{en:dist_zero_zero} Any two consecutive zeros of $\psi_n$ within the classical region have distance $\lesssim_{\kappa} n^{-2/3} |\{V \leq E_n\}|$
\end{enumerate}

By using the above information, we can improve, for all $V \in \pot_1(\kappa)$, the uniform estimate \eqref{eq:uniform_psi_rough}. 
Indeed, we know that the maximum of $|\psi_n|$ is attained at one of the outermost zeros of $\psi_n'$; let us call this point $w_n$, and let $\zeta_n$ be the outermost zero of $\psi_n$ with the same sign. From the above discussion, we deduce that, for all $n \geq n_0$, $|\zeta_n-w_n| \lesssim_{\kappa} n^{-2/3} |\{V \leq E_n\}| \simeq_\kappa n^{1/3}/ E_n^{1/2}$, hence, by \eqref{eq:uniform_der},
\begin{equation}\label{eq:uniform_psi_weak}
\|\psi_n\|_\infty \leq \left| \int_{\zeta_n}^{w_n} \psi_n' \right| \leq |\zeta_n - w_n| \|\psi_n'\|_\infty \lesssim_{\kappa} \frac{n^{1/3}}{|\{V \leq E_n\}|^{1/2}} .
\end{equation}
The same estimate for $n < n_0$ is already contained in \eqref{eq:uniform_psi_rough}, since $n \simeq_{\kappa} n^{1/3}$ for $n < n_0$. Combining this estimate with \eqref{eq:eigen_pointwise_weak} proves the validity of the estimate for $\psi_n$ of Theorem \ref{thm:pointwise_est}\ref{en:pointwise_est_weak} within the classical region.

The above information on the location of the zeros and extrema of $\psi_n$ and $\psi_n'$ can also be used to prove the estimates for $\psi_n'$ of parts \ref{en:pointwise_est_new} to \ref{en:pointwise_est_c3} of Theorem \ref{thm:pointwise_est} on the whole real line, as well as the corresponding estimates for $\psi_n$ in the classical region. Indeed, under the assumptions on $V$ and $\alpha$ in any of parts \ref{en:pointwise_est_new} to \ref{en:pointwise_est_c3} of Theorem \ref{thm:pointwise_est}, we know from \eqref{eq:est_eigen_pointwise} and \eqref{eq:est_eigen_pointwise_new} that the improved pointwise estimate
\begin{equation}\label{eq:est_eigen_pointwise_joint}
\psi_n^2 + \left(1 + \frac{n^{-2/3} E_n}{E_n-V}\right)^{-3} \frac{(\psi_n')^2}{E_n-V} \lesssim_{\bar\kappa,\alpha} \frac{E_n^{2\alpha} (E_n-V)^{-2\alpha}}{|\{V \leq E_n\}|}
\end{equation}
holds in the classical region $\{V < E_n\}$. In particular,
\begin{equation}\label{eq:est_eigen_pointwise_joint_der}
(\psi_n')^2 \lesssim_{\bar\kappa,\alpha} \frac{E_n^{2\alpha} (E_n-V)^{1-2\alpha}}{|\{V \leq E_n\}|} \qquad\text{on } \{ V \leq E_n - E_n/n^{2/3} \}.
\end{equation}
If we now apply this estimate 
at the two outermost zeros $\pm z_n^\pm$ of $\psi_n$ within the region $\{ V \leq E_n - E_n/n^{2/3} \}$, we obtain that, for all $n \geq n_0$,
\[
\psi_n'(\pm z_n^\pm)^2 \lesssim_{\bar\kappa,\alpha} \frac{E_n^{2\alpha} (E_n-V(\pm z_n^\pm))^{1-2\alpha}}{|\{V \leq E_n\}|} \simeq_{\kappa} \frac{E_n}{|\{V \leq E_n\}|} n^{-2(1-2\alpha)/3},
\]
where we used that $E_n-V(\pm z_n^\pm) \simeq_{\kappa} n^{-2/3} E_n$, due to the fact that the distance between $\pm z_n^\pm$ and the transition point of the same sign is $\simeq_{\kappa} n^{-2/3} |\{V \leq E_n\}|$ (see \ref{en:dist_zero_tran} and \ref{en:dist_zero_zero} above).

As previously discussed, the $\pm z_n^\pm$ are local maximum points of $(\psi_n')^2$, and because of the monotonicity properties of the sequence of local maxima of $(\psi_n')^2$ (see \ref{en:monot_dermax} above), we conclude that, for all $n \geq n_0$,
\begin{equation}\label{eq:est_der_outside}
(\psi_n')^2 \leq \max_{\pm} \psi_n'(\pm z_n^\pm)^2 \lesssim_{\bar\kappa,\alpha} \frac{E_n}{|\{V \leq E_n\}|} n^{-2(1-2\alpha)/3} \quad\text{on } \RR \setminus (-z_n^-,z_n^+);
\end{equation}
here we can go beyond the classical region, because $(\psi_n')^2$ is increasing on $\{V \geq E_n\} \cap (-\infty,0)$ and decreasing on $\{V \geq E_n\} \cap (0,\infty)$ (see \ref{en:monot_outside} above). By combining \eqref{eq:est_eigen_pointwise_joint_der} and \eqref{eq:est_der_outside} we deduce the estimates for $\psi_n'$ of Theorem \ref{thm:pointwise_est}\ref{en:pointwise_est_new}-\ref{en:pointwise_est_c3} on the whole $\RR$. 

As for the bound on $\psi_n$, we can argue as in \eqref{eq:uniform_psi_weak}, but use the improved bound on $\psi_n'$ from \eqref{eq:est_der_outside}. Namely, let $w_n$ be an outermost zero of $\psi_n'$ where $|\psi_n|$ attains its maximum, and $\zeta_n$ be the outermost zero of $\psi_n$ with the same sign. Then
$\zeta_n,w_n \in \RR \setminus (-z_n^-,z_n^+)$ and $|z_n-w_n| \lesssim_{\kappa} n^{-2/3} |\{V \leq E_n\}| \simeq_\kappa n^{1/3}/E_n^{1/2}$, so
\begin{equation}\label{eq:uniform_psi}
\|\psi_n\|_\infty = |\psi_n(w_n)| \leq \left|\int_{\zeta_n}^{w_n} \psi_n'\right| \lesssim_{\kappa} \frac{n^{2\alpha/3}}{|\{V \leq E_n\}|^{1/2}}
\end{equation}
for all $n \geq n_0$. The same estimate for $n < n_0$ is already contained in \eqref{eq:uniform_psi_rough}, since $n \simeq_{\kappa,\alpha} n^{2\alpha/3}$ for $n < n_0$. Combining this estimate with \eqref{eq:est_eigen_pointwise_joint} proves the validity of the estimates for $\psi_n$ of Theorem \ref{thm:pointwise_est}\ref{en:pointwise_est_new}-\ref{en:pointwise_est_c3} within the classical region.

\subsection{Estimate outside the classical region}\label{ss:exponential}
In order to complete the proof of Theorem \ref{thm:pointwise_est},
it remains to prove
the estimate
\[
|\psi_n(x)| \lesssim_{\bar\kappa,\alpha} \frac{1}{|\{V \leq E_n\}|^{1/2}} (V(x)/E_n-1)^{-\alpha}  
\]
outside the classical region. 
We only discuss the case $x>0$, since the case $x<0$ is treated analogously. So we need to prove the above estimates for $x>x_n^+$, where $x_n^+$ is the positive transition point.

Here we use the estimate from \cite[\S 8.2]{titchmarsh},
\[
|\psi_n(x)| \leq |\psi_n(x_n^+)| \exp\left(-\int_{x_n^+}^x (V-E_n)^{1/2}\right) \leq \|\psi_n\|_\infty \exp\left(-\int_{x_n^+}^x (V-E_n)^{1/2}\right) ,
\]
valid for all $x \geq x_n^+$, together with the estimates for $\|\psi_n\|_\infty$ obtained previously.

Let $\tilde x_n^+ > 0$ be such that $V(\tilde x_n^+) = 4E_n$. Then, for $x \geq \tilde x_n^+$,
\[
\int_{x_n^+}^x (V-E_n)^{1/2} \simeq_{\kappa} x \sqrt{V(x)} \geq x_n^+ E_n^{1/2} \sqrt{V(x)/E_n} \simeq_{\kappa} n \sqrt{V(x)/E_n}
\]
(see \cite[eq.\ (6.11)]{DM2} and Proposition \ref{prp:BS_rough}) and both factors in the last product are greater than or equal to $1$. Hence, for some $c=c(\kappa)$, if we use the estimate for $\|\psi_n\|_\infty$ from \eqref{eq:uniform_psi_rough}, then we deduce that
\[\begin{split}
|\psi_n(x)| &\lesssim_{\kappa} \frac{n}{|\{V \leq E_n\}|^{1/2}} \exp(-cn) \exp(-c \sqrt{V(x)/E_n}) \\
&\lesssim_{\kappa,N} \frac{1}{|\{V \leq E_n\}|^{1/2}} (V(x)/E_n)^{-N}  \simeq_N \frac{1}{|\{V \leq E_n\}|^{1/2}} (V(x)/E_n-1)^{-N},
\end{split}\]
for any $N > 0$, since $V(x)/E_n \geq 4$ for $x \geq \tilde x_n^+$.

For $x \in (x_n^+,\tilde x_n^+)$, instead,
\[\begin{split}
\int_{x_n^+}^x (V-E_n)^{1/2} &\simeq_{\kappa} \frac{|\{V \leq E_n\}|}{E_n} \int_{x_n^+}^x (V-E_n)^{1/2} V' \\
&\simeq_{\kappa} \frac{|\{V \leq E_n\}|}{E_n} (V(x)-E_n)^{3/2} \\
&\simeq_{\kappa} n\,  (V(x)/E_n-1)^{3/2},
\end{split}\]
by Proposition \ref{prp:BS_rough}. So, if we use the estimate for $\|\psi_n\|_\infty$ from \eqref{eq:uniform_psi_weak} and \eqref{eq:uniform_psi}, then
\[\begin{split}
|\psi_n(x)| &\lesssim_{\bar\kappa,\alpha}  \frac{n^{2\alpha/3}}{|\{V \leq E_n\}|^{1/2}} \exp\left(-c n \, (V(x)/E_n-1)^{3/2}\right) \\
&\lesssim_\alpha \frac{1}{|\{V \leq E_n\}|^{1/2}} (V(x)/E_n-1)^{-\alpha},  
\end{split}\]
as desired.

\section{Proof of the sharpened weighted Plancherel estimate}\label{s:weightedplancherel}

We are finally in a position to prove the desired sharpened version of the weighted Plancherel estimate of \cite[Theorem 9.1]{DM2}. We restate it as a separate theorem.

\begin{thm}\label{thm:weightedplancherel_l2}
Assume that $V \in \pot_{1+\theta}(\kappa)$ for some $\theta \in (0,1)$.
Let $\mm : \RR \to \CC$ be a bounded Borel function such that $\supp \mm \subseteq [1/4,1]$. 
Then, 
for all $\vartheta \in [0,1/2)$ and all $r>0$,
\begin{multline*}
\esssup_{z' \in \RR^2} \, r^{2-2\vartheta} \max\{V(r),V(x')\}^{1/2-\vartheta} \int_{\RR^2} |y-y'|^{2\vartheta} \left|\Kern_{\mm(r^2 \opL)}(z,z') \right|^2 \,dz \\
\lesssim_{\theta,\kappa,\vartheta} \|\mm\|_{\sobolev{\vartheta}{2}}^2,
\end{multline*}
where $z=(x,y)$ and $z'=(x',y')$.
\end{thm}


\begin{proof}[Proof of Theorem \ref{thm:weightedplancherel_l2}]
We follow the set-up and notation of \cite[Section 9]{DM2}, but assume additionally that
 $\supp \mm \subseteq [1/4,1]$.
For $A \in \Rpos$, define $G_A(\lambda,\tau) = \mm(\lambda) \chi(A\tau)$, where $\chi \in C^\infty_c([1/4,1])$ is such that $\sum_{j \in \ZZ} \chi(2^j \cdot) = 1$ on $\Rpos$, and 
 set $K_{G_A} = \Kern_{G_A(\opL,-\partial_y^2)}$. Then
\begin{equation}\label{eq:weighted_basic}
\begin{aligned}
\int_{\RR^2} |K_{G_A}(z',z)|^2 \,dz' &\lesssim A^{-1/2} \int_{A^{-1}}^{4A^{-1}} \| \matM_1(\tau) . \vec\psi(x;\tau V) \|^2 \frac{d\tau}{\tau} \\
\int_{\RR^2} (y'-y)^2 \, |K_{G_A}(z',z)|^2 \,dz' &\lesssim A^{1/2} \sum_{j=1}^4 \int_{A^{-1}}^{4A^{-1}} \| \matM_j(\tau) . \vec\psi(x;\tau V) \|^2 \frac{d\tau}{\tau}
\end{aligned}
\end{equation}
(see \cite[eqs.\ (9.11) and (9.15)]{DM2}), where
\begin{align*}
\matM_1(\tau) &= \diag (\mm (\vec E(\tau V))),\\
\matM_2(\tau) &= \diag \mm'(\vec{E}(\tau V)) \schur \diag \vec{F}(\tau V), \\
\matM_3(\tau) &= \matN \schur \matA(\tau V) \schur \inc{\mm(\vec{E}(\tau V))} , \\
\matM_4(\tau) &= \matF \schur \matA(\tau V) \schur \inc{\mm(\vec{E}(\tau V))} . 
\end{align*}
Here $\vec E(\tau V) = (E_n(\tau V))_n$, $\vec F(\tau V) = (\tau \partial_\tau E_n(\tau V))_n$, $\vec \psi(\cdot;\tau V) = (\psi_n(\cdot;\tau V))_n$; the matrices $\matA(\tau V)$, $\matP(\tau V)$, $\matN$, $\matF$, are given by $\matA_{nm}(\tau V) = \langle \tau \partial_\tau \psi_n(\cdot;\tau V), \psi_m(\cdot;\tau V) \rangle$, $\matP_{nm}(\tau V) = \langle \tau V \psi_n(\cdot; \tau V), \psi_m(\cdot;\tau V) \rangle$, $\matN_{nm} = \chr_{n/2 \leq m \leq 2n}$, $\matF_{nm} = \chr_{n > 2m} + \chr_{m > 2n}$; moreover $\odot$ is the Schur product between matrices, $\|\cdot\|$ is the $\ell^2$ norm, and
\[
\diag \vec f = (f_n \delta_{nm})_{n,m}, \qquad \inc \vec f = (f_n-f_m)_{n,m}.
\]

In the proof of \cite[Theorem 9.1]{DM2}, the integrals in $\frac{d\tau}{\tau}$ in \eqref{eq:weighted_basic} are bounded by the corresponding suprema, which eventually results in estimates involving $L^\infty$ Sobolev norms of $\mm$. In order to obtain sharper estimates with $L^2$ Sobolev norms, here instead we crucially take advantage of the integration in $\tau$.

Let us first consider the term involving the diagonal matrix $\matM_1(\tau)$:
\[\begin{split}
&\int_{A^{-1}}^{4A^{-1}} \| \matM_1(\tau) . \vec\psi(x;\tau V) \|^2 \frac{d\tau}{\tau} \\
&= \int_{A^{-1}}^{4A^{-1}} \sum_n |\mm(E_n(\tau V))|^2 \psi_n(x;\tau V)^2 \frac{d\tau}{\tau} \\
&= \int_0^\infty |\mm(\lambda)|^2 \sum_{n \tc  1 / \Xi_n(\lambda) \in [A/4,A]} \psi_n(x;\Xi_n(\lambda) V)^2 \frac{\lambda \Xi_n'(\lambda)}{\Xi_n(\lambda)} \frac{d\lambda}{\lambda} \\
&\lesssim_{\kappa} \int_0^\infty |\mm(\lambda)|^2 \sum_{n \tc  \lambda / \Xi_n(\lambda) \in [A/16,A]} \psi_n(x;\Xi_n(\lambda) V)^2  \frac{d\lambda}{\lambda}
\end{split}\]
by Proposition \ref{prp:inv_eigenvalue}. In light of Theorem \ref{thm:pointwise_est}\ref{en:pointwise_est_new} and Remark \ref{rem:hoelder_uc}, we can apply Theorem \ref{thm:new_estimate} with $\tilde\pot = \pot_{1+\theta}(\kappa)$ to bound the above sum, and deduce that
\begin{equation}\label{eq:est_diagonal}
\int_{A^{-1}}^{4A^{-1}} \| \matM_1(\tau) . \vec\psi(x;\tau V) \|^2 \frac{d\tau}{\tau} \lesssim_{\theta,\kappa} \|\mm\|_2^2 \, (\chr_{V \leq 4A} + e^{-c |x|} \chr_{V > 4A}).
\end{equation}

As observed in \cite[Section 9.3]{DM2}, $\matM_2(\tau)$ is also a diagonal matrix, with diagonal entry $\mm'(E_n(\tau V)) \tau \partial_\tau E_n(\tau V)$, and
\[
|\mm'(E_n(\tau V)) \tau \partial_\tau E_n(\tau V)| \leq |\widetilde \mm (E_n(\tau))|,
\]
where $\widetilde m(\lambda) = \lambda m'(\lambda)$ (see Proposition \ref{prp:inv_eigenvalue}). So the same argument as above, with $\widetilde\mm$ in place of $\mm$, yields
\begin{equation}\label{eq:est_diagonal2}
\int_{A^{-1}}^{4A^{-1}} \| \matM_2(\tau) . \vec\psi(x;\tau V) \|^2 \frac{d\tau}{\tau} \lesssim_{\theta,\kappa} \|\mm'\|_2^2 \, (\chr_{V \leq 4A} + e^{-c |x|} \chr_{V > 4A}).
\end{equation}

We now deal with the ``near-diagonal'' term $\matM_3(\tau)$. As discussed in \cite[Section 9.5]{DM2}, the absolute value of the $(n,m)$ entry of $\matM_3(\tau)$ is
\[
\chr_{n/2 \leq m \leq 2n} \, |\matA_{nm}(\tau V)| \, |\mm(E_n(\tau V)) - \mm(E_m(\tau V))|.
\]
Note that, since $\supp \mm \subseteq [1/4,1]$ and $E_n(\tau V) \simeq_{\kappa} E_m(\tau V)$ for $n/2 \leq m \leq 2 n$, there is $S = S(\kappa) \geq 1$ such that the above entry vanishes unless $E_m(\tau V) \in [S^{-1},S]$. Now,
\[\begin{split}
| \mm(E_n(\tau V)) - \mm(E_m(\tau V)) | &= \left| \int_{E_m(\tau V)}^{E_n(\tau V)} \mm'(\lambda) \,d\lambda \right| \\
&\leq |E_n(\tau V) - E_m(\tau V)| \, \HLM (\mm') (E_m(\tau V)),
\end{split}\]
where $\HLM$ denotes the uncentred Hardy--Littlewood maximal function on $\RR$. Consequently, if we set $\widehat m = \chr_{[S^{-1},S]} \HLM \mm'$, then
\[\begin{split}
&\chr_{n/2 \leq m \leq 2n} \, |\matA_{nm}(\tau V)| \, |\mm(E_n(\tau V)) - \mm(E_m(\tau V))| \\
 &\leq \chr_{n/2 \leq m \leq 2n} \, |\matP_{nm}(\tau V)| \,  \widehat\mm(E_m(\tau V))  \\
&\lesssim_{\kappa,\theta} \frac{1}{1+|m-n|^{1+\epsilon}} \,  \widehat\mm(E_m(\tau V)) ,
\end{split}\]
where $\epsilon = \epsilon(\kappa,\theta)$, and we applied \cite[Proposition 8.1 and Theorem 8.4]{DM2}. Since the matrix $((1+|m-n|^{1+\epsilon})^{-1})_{n,m \geq 1}$ is $\ell^2$-bounded, we conclude that
\[
\| \matM_3(\tau) . \vec\psi(x;\tau V) \| \lesssim_{\kappa,\theta} \| \diag(\widehat\mm(\vec E(\tau V))) . \vec\psi(x;\tau V) \|.
\]
So, the same argument that proves \eqref{eq:est_diagonal}, applied with $\widehat \mm$ in place of $\mm$, yields, for some $T_1 = T_1(\kappa)$,
\begin{equation}\label{eq:est_neardiag}
\begin{split}
\int_{4A^{-1}}^{A^{-1}} \| \matM_3(\tau) . \vec\psi(x;\tau V) \|^2 \frac{d\tau}{\tau} 
&\lesssim_{\theta,\kappa} \| \widehat\mm \|_2^2  \, (\chr_{V \leq T_1 A} + e^{-c |x|} \chr_{V > T_1 A}) \\
&\lesssim \| \mm' \|_2^2  \, (\chr_{V \leq T_1 A} + e^{-c |x|} \chr_{V > T_1 A}),
\end{split}
\end{equation}
where the last bound follows from the $L^2$-boundedness of $\HLM$.

Finally, from \cite[eq.\ (9.19)]{DM2}, we already know that
\begin{equation}\label{eq:est_fardiag}
\begin{split}
\int_{A^{-1}}^{4A^{-1}} \| \matM_4(\tau) . \vec\psi(x;\tau V) \| \frac{d\tau}{\tau} &\lesssim_{\kappa} \|\mm\|_\infty^2 (\chr_{V \leq T_2 A} + e^{-c |x|} \chr_{V > T_2 A}) \\
&\lesssim (\|\mm\|_2 + \|\mm'\|_2)^2 (\chr_{V \leq T_2 A} + e^{-c |x|} \chr_{V > T_2 A})
\end{split}
\end{equation}
for some $T_2 = T_2(\kappa) \geq 4$, where the last estimate follows from Sobolev's embedding.

In conclusion, from \eqref{eq:weighted_basic}, \eqref{eq:est_diagonal}, \eqref{eq:est_diagonal2}, \eqref{eq:est_neardiag} and \eqref{eq:est_fardiag}, we deduce that
\[\begin{aligned}
\int_{\RR^2} |K_{G_A}(z',z)|^2 \,dz' &\lesssim_\kappa A^{-1/2} \| \mm \|_2^2  \, (\chr_{V \leq T_3 A} + e^{-c |x|} \chr_{V > T_3 A}), \\
\int_{\RR^2} (y'-y)^2 \, |K_{G_A}(z',z)|^2 \,dz' &\lesssim_{\kappa,\theta} A^{1/2} (\| \mm \|_2 + \|\mm'\|_2)^2  \, (\chr_{V \leq T_3 A} + e^{-c |x|} \chr_{V > T_3 A}),
\end{aligned}\]
where $T_3 = \max\{T_1,T_2\}$.
These two estimates are analogous to the ones stated at the beginning of \cite[Section 9.6]{DM2}, with $L^2$ norms of $\mm$ and $\mm'$ instead of $L^\infty$ norms. As in \cite[Section 9.6]{DM2}, by interpolating these two estimates and then summing them for $A = 2^j$, $j \in \ZZ$, one eventually deduces that, for all $\vartheta \in [0,1/2)$,
\[
\left(\int_{\RR^2} |y-y'|^{2\vartheta} |\Kern_{m(\opL)}(z',z)|^2 \,dz'\right)^{1/2} \lesssim_{\theta,\kappa,\vartheta} \|m\|_{\sobolev{2}{\vartheta}} \max\{V(1),V(x)\}^{\frac{\vartheta}{2}-\frac{1}{4}},
\]
which is the case $r=1$ of the estimate in Theorem \ref{thm:weightedplancherel_l2}. The estimate for arbitrary $r>0$ follows by rescaling, that is, by replacing $V(x)$ with $V_r(x) = r^2 V(rx)$, as explained at the end of \cite[Section 9.6]{DM2}.
\end{proof}

\end{document}